\documentclass[10.5pt]{amsart}
\usepackage[centering]{geometry} 
\usepackage{hyperref}
\usepackage{graphicx}
\usepackage{mathrsfs}
\usepackage{amssymb}
\usepackage{epstopdf}
\usepackage{amscd}
\usepackage{comment}
\usepackage{diagrams}
\usepackage[english]{babel}

\newtheorem{theorem}{Theorem}[section]

\newtheorem{proposition}[theorem]{Proposition}
\newtheorem{corollary}[theorem]{Corollary}

\newtheorem{lemma}[theorem]{Lemma}

\newtheorem{definition}[theorem]{Definition}

\newtheorem*{theorem-non}{Theorem}

\DeclareMathOperator{\Tors}{Tors}

\author{Fedor Bogomolov} 
\address {\tiny{Fedor Bogomolov \newline Courant Institute of Mathematical Sciences,
 New York University, 
 251 Mercer Street, 
 New York, NY 10012, USA\newline
 National Research University Higher School of Economics, Russian Federation,
AG Laboratory, HSE, 6 Usacheva str., Moscow, Russia, 119048}}
\email{bogomolo@cims.nyu.edu}

\author{Nikon Kurnosov}
\address {\tiny{Nikon Kurnosov\newline
 National Research University Higher School of Economics, Russian Federation,
AG Laboratory, HSE, 6 Usacheva str., Moscow, Russia, 119048 \newline
Department of Mathematics,
University of Georgia,
1023 D. W. Brooks Drive,
Athens, GA 30602, USA}}
\email{nikon.kurnosov@gmail.com}

\title{Lagrangian fibrations for IHS fourfolds}
\begin{document}
\begin{abstract}

In this paper we study the Lagrangian fibrations for projective irreducible symplectic fourfolds and exclude the case of non-smooth base. Our method could be extended to the higher-dimensional cases.

\end{abstract}

\maketitle

\section{Introduction}

A variety $M$ of dimension $2n$ endowned with an everywhere non-degenerate closed two-form $\Omega \in H^0 (M, \Omega^2(M))$ such that $\Omega^n$ is the section of anticanonical class is called a holomorphic symplectic variety. Compact simply-connected holomorphic symplectic manifolds such that $H^{2,0}(M, \mathbf{C})=\mathbf{C} \left\langle \Omega \right\rangle $ are called irreducible holomorphic symplectic (IHS) manifold. Well-known that K\"ahler holomorphic symplectic manifolds are hyperk\"ahler and what by Bogomolov theorem \cite{B} any hyperk\"ahler manifold manifold is covered by product torus and irreducible ones. 

A subvariety $N \subseteq M$ of dimension $n$ is said to be Lagrangian if the
restriction of $\Omega$ on the smooth locus of $N$ is identically zero.

\begin{theorem}(Matsushita, \cite{Ma1}).
Let $\pi:\; M \rightarrow B$ be a surjective holomorphic map
from a hyperk\"ahler manifold $M$ to a base $B$, with $0<\dim B < \dim M$.
Then $\dim B = 1/2 \dim M$, and the fibers of $\pi$ are 
holomorphic Lagrangian (this means that the symplectic
form vanishes on the fibers). Moreover, $B$ is $\mathbf{Q}$-factorial klt Fano variety of dimension $n$ with Picard number 1.

\end{theorem}

\begin{definition} Such a map is called {\bf a holomorphic Lagrangian fibration}.
\end{definition}

Consider a Lagrangian fibration from a complex projective irreducible symplectic manifold M to a normal
variety $B$. In all known examples of Lagrangian fibrations, the base $B$ is always
isomorphic to $\mathbf{C} P^n$. If one
assume that the base is smooth, then Hwang (\cite{Hwa}) proved that the base is always
a projective space. On the other hand, if the irreducible symplectic manifold $M$ is of $K3^{[n]}$ or generalized
Kummer deformation types, then the answer is also positive (see \cite{Ma2},\cite{Mar}, \cite{BM} and \cite{Yos}). However, it is still unclear if the base is always smooth, for instance for O'Grady sporadic examples. Thus the answer to this question is still unknown in the general setting. Quite recently the following theorem has been proved providing just two possibilities left

\begin{theorem}(Ou, \cite{Ou})
Let $\pi : M \rightarrow X$ be a Lagrangian fibration from a projective IHS fourfold $M$ to a normal surface $X$. Then either $X$ is either $\mathbf{C} P^2$ or $S^n(E_8)$, where the surface $S^n(E_8)$ is the unique Fano surface with exactly one singular point which is Du Val of type $E_8$, and two nodal rational curves in its anti-canonical system.
\end{theorem}

In this paper we focuse in the case when $M$ has dimension 4 and prove that the latter case is impossible and further get Matsushita's Conjecture in that case:

\begin{theorem}\label{main}
Let $\pi:\; M \rightarrow X$ be a Lagrangian fibration from a projective IHS fourfold $M$ to a normal surface $X$. Then $X \simeq \mathbf C P^2$.
\end{theorem}

Moreover, our proof provides an approach to the general case.

\section{Lagrangian fibrations of IHS fourfolds}

In this section we prove the Main Theorem \ref{main}. Idea of the proof is the following: we consider fundamental group of fiber over singular point, which is abelian, and of its neighborhood, one have surjective map from it to fundamental group of complement of singular point to the base. The description of singular fibers were given in projective case by Matsushita \cite{Ma1} and in the general case by Hwang and Oguiso \cite{HO}. This provides us restrictions for the possible types of singularities.

\begin{lemma} \label{surj}
Let $\pi:\; X \rightarrow B$ be a Lagrangian fibration from a projective IHS fourfold $M$ to a normal surface $X$ with one Du Val singularity $b_0$. Let $U(X_0)$ be a small neighborhood of the fiber $X_0$, and $V(b_0)$ a small neighborhood of $b_0$.  Then there is a surjection of $\pi_1( U(X_0))$ onto $\pi_1(V(b_0)\setminus b_0)$.
\end{lemma}


\proof

Since the map $X\to B$ is flat and  the base $B$ is normal, then all the fibers have codimemsion $2$. Let $b_0\in B$ be a singular point.
Then by Mumford the local fundamental group $\pi_1(V(b_0)\setminus b_0)$ is non-trivial. Consider the transversal disc $D$ to the fiber, then there is a map of it to the smooth point of a fiber over a singular point $b_0$. This map is surjective onto a neighborhood of $b_0$, therefore $\pi_1(V(b_0)\setminus b_0)$ is finite, and hence by classification is a subgroup of
$GL(2,\mathbf{C})$, and local neighborhood is isomorphic to $D/\pi_1(V(b_0)\setminus b_0)$.

Generic fiber is irreducible, thus any circle in $V(b_0) \setminus b_0$ can be lifted isomorphically
to the circle in $X \setminus X_0$. Then previous argument implies that $\pi_1( U(X_0))$ surjects onto
$\pi_1(V(b_0)\setminus b_0)$.
 \endproof

Now we will proof the Main Theorem \ref{main}. Let us resolve the local singularity by finite covering $D\to D/G \simeq V(b_0)$. Assume that the fiber over the zero which is unique preimage of $b_0$ has multiplicity one. Let $Y_0$ be the fiber after covering. It is non-ramified covering over singular point with the group $G$ acting freely on a fiber.

\proof {of \ref{main}}
      
Assume first that after lifting to the smooth disc $D$ we have a fibration without multiple fibers. And the group acts freely on the central fiber $Y_0$ over $0$, which projects to  $0 \in D$, $Y_0/ G \simeq X_0$.

Then we have an abelian fibration over smooth disc $D$ induced from $B$ under the map $D \to D /G$ isomorphic to $B$.      
      
In the case of $E_8$-type singularity the fundamental group $G =\pi_1(V(b_0)\setminus b_0)$ is isomorphic to a unique non-trivial central extension. 

$$1\to \mathbf{Z}/2\to \tilde{A_5} \to A_5 \to 1.$$

The latter has a representation on $\mathbf{C}^2$ with free action of $\mathbf{C}^2 \setminus 0$, i.e. representation on tangent space $T_0(D)$ of disc $D$ at 0. 

\begin{displaymath}
\begin{diagram}
Y& \rTo & D &\ni 0 \\
\dTo &  & \dTo^{\rho}  & \\
X& \rTo & D/\tilde A_5 & \ni b_0\\ ,
\end{diagram}
\end{displaymath}

If the fiber $Y_0$ has components of multiplicity one, then the union of such components $Y_0^1$ contains the commutative algebraic subgroup $A_0$, which consists of smooth points $Y_0^1$, and the closure of $A_0$ is $Y_0^1$. The action of $G$ maps $A_0$, $Y_0^1$ into themself. The action of $G$ on $A_0$ is affine, namely, for any $x,y$: $g(x-y)=gx-gy$. This expression does not depend on the choice of  $0$ in $A_0$. 

If any non-trivial element $g \in G$ maps a component of $Y_0^1$ into itself, then it has an invariant point in the closure $Y_0^1$, which contradicts the fact that the action of $G$ on $Y_0$ is free.

Indeed, the action of $G$ on an invariant vector fields in tangent space  of the component is dual via non-degenerate 2-form to the tangent space $T_0(D)$ of $D$ at 0. The action of $G$ at the latter is the standart 2-dimensional  representation $\tilde{A_5}$, and has a property that non-identity elements have eigenvalue one. For any $g \in G$ its action on the group components lifts to the action on univeral covering, which is $\mathbf{C}^2$ with natural affine structure. Since $g$ has no eigenvalue one, it has a atable point on the universal covering of the corresponding connected component of the group $A_0$, which is the linear space. Hence, $g$ also has an invariant point on $Y_0$, providing a contradiction with freeness of the action. 

Therefore, if $Y_0^1$ is non-trivial, we have to consider an affine action on the group of components.

Let us first consider the case if the multiplicity of some components of the fiber is one. Hence $Y_0^1$ is non-trivial.

\begin{proposition}
The action of $\tilde A_5$ on the  the fiber $Y_0^1$
  can not be free.
\end{proposition}

\begin{proof}


The group $T$ of components is a finite subgroup of a generic abelian fiber $A^2$ which is $\mathbb{(\mathbf{Q}/\mathbf{Z})}^{4}$.


Note that the Euler pairing is trivial on $T$, i.e $T$ is isotropic subgroup
with respect to this skew-symmetric pairing.
In particular $T$ has at most two generators.
The group $G$ acts on $T$
affinely,i.e $g$ acts by an automorphism on the
elements $x-x_0$, i.e degree $0$ combinations.
Thus
 $$g(x)= lg(x) + b_g $$ 
for a selected $0$-element
in $T$.

The group $T$ is a direct sum of a finite $p$-groups $T_{p}$ of elements
of order a power of $p$. Thus we have an embedding into affine action on the product of the groups $T_p$. Each of those groups by assumption has at most two generators.

The group of affine automorphism of $T_p$ is an extension of a subgroup of $Gl_2(\mathbf{Z}_p)$ by a $p$-group.
Note that we have an equivariant $\tilde{A_5}$-action with respect to the
projection to $T\to T_{p}$. If $p \neq 2,3,5$ and hence coprime to the order of $\tilde{A_5}$, then the action is linear. Therefore we can restrict our analysis to the preimage of stable point in $T_p$. Thus, we can reduce the gemeral case to $p=2,3,5$. Note that any map of $\tilde{A_5}$ into $Gl_2(\mathbf{Z}_p)$, where $p=2,3$ is trivial. Hence, we can reduce to the case $p=5$.  In case of $p=5$ we know that the a linear group projects surjectively to $GL_2(\mathbf{Z}_5) = S_5$, which contains $A_5$. But its central extension can not be lifted to $GL_2(\mathbf{Z}_5)$. Note that this subgroup $A_5$ of $S_5$ has only linear lifting in the corresponding affine group acting on $T_5$.
\end{proof}

If the multiplicity is not one, then we can use results of Hwang-Oguiso. Assume now that the fiber over $0$ is multiple.
  In \cite{HO} such fibers are classified using the results
  from \cite{HO1}.
  Namely we summarize them in the following corollary of the \cite[Theorem 1.1]{HO} for the case of $dim M =4$

\begin{corollary} Let $p : M\to B^2$ be an abelian fibration with
  multiple fibers along the divisor $D\subset B$ in a smooth ball $B^2$.
Then
  the normalization $V_x$ of the fiber over generic point $x$ of $D$ is  a
fibration over
   an elliptic curve $E_x, f_x: V_x\to E_x$.The map of the normal bundle
$N(V)\to p^*T(B^2)$ has rank
   $1$ at a smooth point of $V_x$
   and it's kernel is contained in the tangent bundle of the
one-dimensional fiber
   of $f_x$.
   The fiber of $f_x$ is called a characteristic cycle and it is an
elliptic
   curve or it's Kodaira degeneration of the one of the following types:

\begin{itemize}

 \item[1)] elliptic curve $E_0$,

 \item[2)]  Kodaira type II,

 \item[3)]  Kodaira type $E_1$ type IV, 

 \item[4)]  Kodaira type $E_0$ type III,

 \item[5)]  $I_{2m}$ - ring, $I^*_0$- elliptic curve , $IV, IV^*$.
\end{itemize}  
\end{corollary}

  \begin{lemma} The group $\tilde A_5$ can not act freely on $X_0$ if it is
  a multiple fiber.
\end{lemma}

  \begin{proof}
  
 Consider the action of $\tilde A_5$ on elliptic curve $E_x$.
  induced from the action of $\tilde A_5$ on $X_0$.
  Then note that neither $\tilde{A_5}$ nor $A_5$  has an effective action on an elliptic curve. Hence it is trivial and, therefore, $\tilde A_5$ maps fibers of $f_x$
  into themselves.

  Note that the characteristic cycle can not be a smooth elliptic
  curve by the argument showing the absence of non-trivial action of $\tilde{A_5}$ on elliptic curves. If the characteristic cycle is a tree af a rational curves as in the all above cases apart $I_{2m}$, then any element of $\tilde{A_5}$ has an invariant rational component and hence has the stable point on such a component contradicting the fact that the action has to be free on the fiber $X_0$.

  Thus the only case to consider is when the fiber is the cycle
  of rational curves of type $I_{2m}$. In this case the smooth part of fiber is the product of $\mathbf{C}^* \times \mathbf{Z}_{2m}$. The action of $\tilde A_5$ is the affine with the respect to this group action. In particular, we must obtain an embedding of $\tilde A_5$ to the affine automorphisms of $\mathbf{Z}_{2m}$, which is the dihedral group. Since there are no non-trivial maps from $\tilde A_5$ to dihedral group, we obtain a contradiction.

\end{proof}

  This completes the proof of our main result.
\endproof

Note that in fact the examples of local symplectic fibrations with multiple
fibers do not occur in projective case and hence we have alternative proof with the following lemma.

\begin{lemma} Let $X$ be a smooth projective variety with trivial
canonical class and $p : X \to B$ be flat surjective map
with an abelian variety with $X_t$ as a generic fiber and $B$ is
normal.
Then there are no multiple fibers over smooth points of
$B$.
\end{lemma}

\begin{proof}

Consider the map of the tangent bundle $T_X$ to $p^* T_B, p_* T_X
\to p^* T_B$. Then
multiple fibers are appear constitute a discriminant divisor $E$ in
$X$.
Consider a general smooth curve $C$ intersecting $E$ transversally.
and the restriction of  the  fibration $X$ upon $C$ which we denote
by $X_C$. Assume that $S_C \subset C$ corresponding to multiple fibers.
of $p_C: X_C \to C$. Thus let $m_p$ be the multiplicity of the fiber at
$p\in C$.

Then the canonical class of $X_S$ is a 

$$p_C^* K(C) +\Sigma_{p\in S}
(m_p-1)X_p= K(X_C),$$

where the summation goes over all multiple fibers.

Note that we have equivalence in the Picard group $m_p X_p = X_t$
and this is the only relation between divisors $X_p$ for different
points $p$. In particular in the presence of multiple fibers
$K(X_C)$ is not induced from the base.

On the other hand for for a generic smooth $C$ the determinant
of normal bundle $det N(C)$ to $C$ is induced from $C$ and hence
the restriction 

$$K(X)\restriction_{X_C}= 0 = K(X_C)+ p_C^*\det N(C)$$

 and hence
$ K(X_C)= - p_C^*\det N(C)$ which provides a contradiction
since $p_C^*\det N(C)$ is induced from $C$.
\end{proof}

\section{Existence of section}

In this section we study the existence of topological and algebraic sections for abelian fibration. Let $X \to С$ be a fibration for smooth compact $X$ to smooth 2-dimensional Riemann surface $C$.

The following general topological lemma provides the criteria of existence for topological sections for the fibrations over smooth 2-dimensional  Riemann surface.

  Consider first more general case.
  Namely let $p : X\to C$ be a surjective map
  of smooth compact orientable manifold onto a two-dimensional
  Riemannian surface.
  Assume that the following properties are satisfied:

  \begin{itemize}

   \item[1)] the map is smooth fibration with generic connected fiber $X_t$ over $C\setminus S$, where $S$ is the finite number of points
   \item[2)] $\pi_1(X_t)$ is abelian
   \item[3)] there is an element in $H_2(X,\mathbf{Z})$ which intersects generic fiber $X_t$ with degree $1$
  \end{itemize}

\begin{lemma} Consider a subgroup $F\subset H_2(X,\mathbf{Z})$ which is generated by the
image from $H_2(X_\tau)$ in $\tau \in C$. Then any element $h\in H_2(X,\mathbf{Z})/F$ which intersect
generic fiber $X_t$ by degree $1$ has a representative which is a
topological section.
\end{lemma}

\begin{proof} Consider the preimage of $S$ in $X$ and denote it by
$X_S$. Note that we have exact homology sequence:

$$ H_2(X_S , \mathbf{Z})\to H_2 (X,\mathbf{Z}) \to H_2( X\setminus X_S,\mathbf{Z})\to H_1(X_S)\to H_1(X) $$

  We obtain an embedding of $H_2(X,\mathbf{Z})/F$ to a subgroup of $H_2(X/X_S,\mathbf{Z})$
  which is a kernel of the boundary $H_2(X/X_S,\mathbf{Z}) \to H_1(X_S)$.
  In particular any element in $H_2(X,\mathbf{Z})/F$ can be realized by a two
  dimensional cycle $R$ in $X$ with degree intersection one with arbitrary
  fiber $X_t$.

   Consider the corresponding homology element in $H_2(X\setminus X_S,\mathbf{Z})=
H_2(X/X_S, \mathbf{Z})$.
   Since it is a smooth fibration we have $H_1(C \setminus S, H_1(X_t, \mathbf{Z})) \oplus \mathbf{Z}$ equal 
to the
   space $H_2(X,\mathbf{Z})/F$ where projection to $\mathbf{Z}$ corresponds to the image in
$H_2(C, \mathbf{Z}) = \mathbf{Z}$.

  Consider the case when $\pi_1(C_S)$ is a free group.
  The group $H_1( C \setminus S, H^1(X_t,\mathbf{Z})$ is equal to the
  the quotient of
  $H_2(\pi_1( X\setminus S,\mathbf{Z})$ by the image of subgroup $H_2(\pi_1(X_t),\mathbf{Z})$.
  Every homomorphism of free group $\pi_1(C\setminus S)\to \pi_1(X\setminus X_S)$
  corresponds to homotopy class of sections in $X\setminus X_S$ over $C\setminus S$.
  Two such sections are homologous 
   if they define the same
  map $H_1(C\setminus S)$ to $H_1(X\setminus X_S)$. Note that the latter
  is exactly equal to $H_1( C\setminus S, H_1(X_t,\mathbf{Z})$.
  In particular any element $h'$ in $H_2( X\setminus X_S, \mathbf{Z})$
  with intersection index $(h', X_t)= 1$ is represented
   a section $\sec(h')$ over $X \setminus X_S$.

Now the homology element $h'\in H_2(X\setminus X_S, \mathbf{Z})$ extends
  to a homology element in $X$ iff the boundary  $dh' = 0$ in $H_1(X_S,\mathbf{Z})$.
  Consider a section $s(h')$ representing $h'$ over $C\setminus S$ and
  small neighborhoods $U(X_i)$ of  fibers $X_i$ over points $s_i\in S$.
  If $dh'= 0$ then the section $\sec(h')$ representing $h'$ intersects the
boundary $dU(X_i)$
  of $U(X_i)$ by a circle $S_i^1$ which homologous to $0$ in $H_1(U (X_i))$.

The boundary $dU(X_i)$ is fibration with fiber $X_t$ over a circle around
$s_i$ in the base.
Thus $\pi_1(dU(X_i)$ is an extension
$ 1\to \pi_1(X_t)\to \pi_1(dU_i)\to \mathbf{Z}e_i$ and  $\pi_1(S_i^1)= \mathbf{Z}e_h$ maps
isomorphically to $\mathbf{Z}e_i$.
Natural embedding induces a surjective map of $\pi_1(dU_i)\to \pi_1(U_i)$

Let us show that the image of $e_h$ is not only homologous to $0$ in
$H_1(U (X_i))$ but
in fact homotopic to zero in  $\pi_1(U (X_i))$.
The group  $\pi_1(dU_i)$ is generated by $\pi_1(X_t)$ and $e_h$. Since by
assumption
$\pi_1(dU_i)$ is abelian $e_h$ maps in the commutator subgroup of
$\pi_1(U_i)$. Consider the kernel $Ker : \pi_1(dU_i)\to pi_1(U_i)$.
It has to be non-trivial and contain an element $k$ projecting to $e_i$
since
otherwise the image of $e_h$ is not homologous to zero in $H_1(U (X_i))$
But any such $k$ and $\pi_1(X_t)$ generate $\pi_1(dU_i)$.
Thus the image of $\pi_1(dU_i)$ in $\pi_1(U_i)$ and hence  $\pi_1(U_i)$
is abelian group. Hence  $e_h$ is homotopic to zero in $\pi_1(U_i)$
and this homotopy extends $s(h')$ into a global section.
If $S = \emptyset$ then image of $H_2(C,\mathbf{Z})$  is primitive, and we obtain $H_1(C \setminus S, H_1(X_t, \mathbf{Z})) \simeq H^2(X,\mathbf{Z})$, that gives an embedding.

In a case of torus we have a group $\mathbf{Z}+\mathbf{Z}$
with $H^2(\mathbf{Z}+\mathbf{Z})$ embedding into $H^2(\pi_1(X)$. Hence the argument
above applies.



Note that if $C \setminus S = S^2$ then $H_2(X,\mathbf{Z})/F= Z$ and the class $h$ as a
above
  is realized as a section over $S^2 \setminus s$ and
  the obstruction to extend it to the section over $S^2$ is trivial if
$H_1(X_t)= 0$.
  If the latter is non-trivial then the obstruction to the extension of it
over $S^2$     corresponds to non-primitivity
  of the image of the $H^2(S^2,\mathbf{Z})$ in $H^2(X,\mathbf{Z})$ which is trivial by
assumption on $p$.
  Namely the existence of $h$ intesecting $X_t$ with index one.
  \end{proof}

\

Here we discuss existence of the section in this case and another way to proof Theorem \ref{main}.

\begin{lemma} \label{rat-section} Let us assume that an abelian fibration $A$ over normal base $B $ of arbitrary dimension has two multisections
 $s_1,s_2$
 of coprime degree $m,n$, then there is rational section.
 \end{lemma}
 
\begin{proof}
Let $a,b$ be integer numbers  with $an+bm = 1$.
Let $A_0$ be associated group scheme of cycles of degree
$0$ over $B$.

 Consider a  map $f: A\to A_0$ with $f(x) =  x - ns_1 + ms_2 $
 which associates to any point in the fiber a point
 in $x\in A_0$ corresponding to cycle given by the sum
 $x$ minus the cycle obtained from the intersection
 of fiber containing $x$ and its intersections with $s_1$ and $s_2$.
 The corresponding cycle has degree $0$ on every fiber and induces
 a fiberwise isomorphism  from $A\to A_0$ which is birational.
 Hence $A$ has a rational section.
\end{proof} 

 We can now prove Hodge type statement for some cycles in $A$
 when $A$ is projective over base.

Recall the following well-known fact:

Consider an abelian fibration $A$ over normal projective base $B$ and the group of sections. Then the latter is the product of $A_0 \times \Tors \times \mathbf{Z}^r$, where $A_0$ is a constant part of $A$ namely $A$ has the unique representation as $A' \times A_0$, and the group of sections of $A'$ is a $\Tors \times \mathbf{Z}^r$. Moreover, on $\mathbf{Z}^r$ there is non-degenerate positive quadratic form defined by polarization (Tate-pairing). Thus, it is defined on topological classes, and, in particular, $\mathbf{Z}^r$ embedds into the homology of $A'$ (and $A$). Hence we have at most  finite set of sections
which belong to the same homology class in $A'$.

\begin{lemma} \label{top-cycle}  Assume that  $B$ is normal and there is exist a topological
 $(p,p)$-cycle $R$ over $B \setminus S$ which intersects generic fiber $A$ by degree $1$
 Then there is a rational section over $B$ which is contained class $R$.
\end{lemma}

\begin{proof}
 Let $C_t$ be a class of curves which are complete intersections
  of big degrees in $B$. Then the restriction of $A$ over $C$ has a rational
  section in the intersection of $R$ with a  restriction $A_C$ of the fibration over $C$
  since the intersection of $R$ with $A_C$ is a class of curve and
  hence be previous \ref{rat-section}  we have a section in class $R_C$.

Note that all the  sections over $C$ form an abelian group with
 non-trivial height function which is a quadratic form.
 The sections $s_R$ in class $R_C$ have a finite order difference
 (we assume that the family on $C$ is non-constant).
 Consider a point $b$ with a smooth abelian fiber and all
 $C$ passing through $b$. The set of all such curves in one class
 $[C]$ forms a rational manifold $M_C$. Moreover
 any two points in $M_C$ can be connected by rational
 curves $P$ isomorphic to $C$ or $P^1$ which do not intersect
 singular subset of $B$. Any such $P$ defines a family
 of cycles $s_R(p),p\in P$ which intersect the fiber $A_b$
 at the same point.
 
 Indeed, the family of cycles $s_R$ in
 over curves parametrizing by $p$ is finite covering of
 $P$ and hence splits. Thus we have
 a rational map $P\to A_b$ which maps $P$
 to one points an hence all
  such cycles pass through the same
point in $A_b$ for any $b$ and form a cycle in the global class
$R$ over $B$.
\end{proof}

Now, over the curve we have two multisections so by \ref{rat-section} we can construct rational section by fiberwise mapping on itself. Then we can apply \ref{top-cycle} to the our case and construct topological section if we consider intersection of circle with the preimage of a general fiber. We can complete our cycle by \ref{top-cycle}.

From the \cite[Theorem 3]{KV} easily follows the following

\begin{corollary}
Let $M$ be a fixed compact manifold of complex dimension $4$ and 
$b_2(M) \geq 7$. Then there are 
only finitely many deformation types of hyperk\"ahler 
Lagrangian fibrations.
\end{corollary}



Note that similar argument can be applied to
higher dimensional projective complex compact
symplectic manifolds with lagrangian fibration
$p: X\to B$ using here and further notations
similar for the $4$-dimensional case considered
above.

In particular if the top dimensional
singularity strata has codimension $2$
in $B$ then its generic transveral-two dimensional
local section $U(b_0)$  has also a singularity
of type $\mathbf{C}^2/G$ where $G$ acts freely
on $\mathbf{C}^2 \setminus 0$.  Thus in projective case
we obtain free action of $G$ on the preimage $Y_0$
of the fiber $X_0$ over $b_0$ after a covering
$D^2 \to D^2/G= U(b_0)$ Since the fiber is not
$Y_0$ is not multiple we obtain that
the action of $G$ on the space of invariant
tangent fields to $Y_0^1$ is a sum of a trivial
action on $\mathbf{C}^{n-2}$ and the above action on $\mathbf{C}^2$.
It implies a decompostion of $Y_0^1$ after finite
  isogeny into a product of and abelian variety
  $A^{n-2}$ with trivial $G$ action  and a
  stable two-dimensional degneration of
  two dimensional abelian variety $Y_2$ with free $G$-action.
  We can then apply previous analysis to $Y_2$
  but unfortunately we can not get the same
  restriction on the size of the finite group
  of components as in case when $\dim X=4$.

  However if the latter is small we can obtain
  the result similar to the case of $\dim X=4$.

  In case of top singularity having codimension $ >2$
  we can apply similar arguments to a larger class
  of groups which appear as fundamental groups
  of $G= \pi_1(U(b_0)\ b_0)$. This class of groups
  is still rather restrictive since every abelian
  subgroup of $G$ has to be cyclic.
  We plan to consider these questions in further
  study.

\
      
      \noindent
\textbf{Acknowledgements.}
Fedor Bogomolov and Nikon Kurnosov acknowledge that the article
was prepared within the framework of a subsidy granted to the HSE by
the Government of the Russian Federation for the implementation of
the Global Competitiveness Program. The second author was partially
supported by RFBR grant 18-31-00146$\setminus$18, Simons travel grant. The first author was partially
supported by EPSRC programme grant EP/M024830, Simons Fellowship
and Simons travel grant.\\

\end{document}